\documentclass{article}
\usepackage{amsthm,amssymb,amsfonts,amsmath,graphicx,cite, color,mathrsfs,booktabs}

\usepackage[margin = 1.5in]{geometry}

\theoremstyle{plain} \numberwithin{equation}{section}
\newtheorem{theorem}{Theorem}[section]
\numberwithin{theorem}{section}
\newtheorem{lemma}[theorem]{Lemma}
\newtheorem{corollary}[theorem]{Corollary}
\newtheorem{proposition}[theorem]{Proposition}
\newtheorem{conjecture}[theorem]{Conjecture}

\theoremstyle{definition}
\newtheorem{definition}[theorem]{Definition}

\newtheorem{remark}[theorem]{Remark}
\newtheorem{example}[theorem]{Example}

\newcommand\cC{\mathcal{C}}
\newcommand\cO{\mathcal{O}}

\newcommand\cD{\mathcal{D}}
\newcommand{\N}{\mathbb{N}}
\newcommand{\R}{\mathbb{R}}
\newcommand{\vertiii}[1]{{\left\vert\kern-0.25ex\left\vert\kern-0.25ex\left\vert #1 \right\vert\kern-0.25ex\right\vert\kern-0.25ex\right\vert}}
\newcommand{\ip}[2]{\left\langle #1 , #2 \right\rangle}    % inner product

\def\transp{^{\text{\sf T}}}

\DeclareMathOperator{\diam}{diam}
\DeclareMathOperator{\affine}{aff}
\DeclareMathOperator{\conv}{conv}

\DeclareMathOperator{\dist}{dist}

\DeclareMathOperator{\vertices}{vert}
\DeclareMathOperator{\faces}{faces}
\DeclareMathOperator*{\argmin}{argmin}

\newcommand{\matr}[1]{\begin{bmatrix} #1 \end{bmatrix}}    % matrix

\def\aa{{\mathbf{a}}}
\def\bb{{\mathbf{b}}}
\def\ee{{\mathbf{e}}}
\def\pp{{\mathbf{p}}}
\def\qq{{\mathbf{q}}}

\def\vv{{\mathbf{v}}}
\def\xx{{\mathbf{x}}}

\def\yy{{\mathbf{y}}}
\def\zz{{\mathbf{z}}}
\def\ww{{\mathbf{w}}}
\def\0{{\mathbf{0}}}

\title{Calculus of the Facial Distance}
\author{Javier Peña~\thanks{{\tt jfp@andrew.cmu.edu}} \and Elias Wirth~\thanks{\tt elias.s.wirth@gmail.com}}

\begin{document}

\maketitle

\begin{abstract}
We develop a few calculus rules to compute or lower bound the facial distance of a polytope.  We illustrate our calculus rules on various popular polytopes.  In particular, we provide a 
lower bound on the facial distance of the Birkhoff polytope.
\end{abstract}

\section{Introduction}

The {\em facial distance} of a polytope, or equivalently its {\em pyramidal width}, is a key condition measure of a polytope that determines the speed of convergence of variants of the Frank-Wolfe algorithm on a polytope.  More precisely, consider a minimization problem of the form
\begin{equation}\label{eq.min.problem}
\min_{\xx\in \cC} \, f(\xx)
\end{equation}
where $f$ is $L$-smooth and $\mu$-strongly convex, and $\cC$ is a polytope.  
The developments in in~\cite{bomze2024frank,braun2025conditional,lacoste2015global,pena2019polytope,
rinaldi2023avoiding,tsuji2022pairwise,wirth2026fast,zhao2026new}
show that several variants of the Frank-Wolfe algorithm applied to~\eqref{eq.min.problem} generate a sequence of iterates $\xx^{(t)}, \; t=0,1,\dots$ such that
$$
f(\xx^{(t)}) - \min_{\xx\in \cC} \, f(\xx) = 
\cO\left(\exp\left( \frac{\mu}{L}\cdot\frac{\Phi(\cC)^2}{\diam(\cC)^2}\cdot t\right)\right)$$
where $\diam(\cC)$ and $\Phi(\cC)$ are respectively the diameter and the facial distance of $\cC$.

The facial distance is also closely related to other interesting problems in convex geometry.  In particular, it is closely related to the work of Deza et al. on {\em kissing polytopes}~\cite{deza2024kissing,deza2025small,deza2026kissing}.
The articles~\cite{lacoste2015global,pena2019polytope} compute the exact value of the facial distance for the following polytopes: the standard simplex, the $\ell_\infty$ unit ball, and the $\ell_1$ unit ball.  Besides these special cases, there is very limited work on computing or even estimating the facial distance of general polytopes.  Indeed, this computational task seems to pose a significant challenge.  

The central goal of this paper is to provide some calculus rules that facilitate the computation of the facial distance.  Our work is inspired by a recent development of Iommazzo et al.~~\cite[Theorem 4.5]{iommazzo2025linear} who derived an expression for the facial distance of the cartesian product of polytopes.  We develop three main calculus rules: a product rule that generalizes~\cite[Theorem 4.5]{iommazzo2025linear}, a rule for obtaining lower bounds in terms of facet-exposing directions, and a rule for computing distances under linear transformations.  We illustrate how these three rules yield exact expressions or lower bounds on the facial distance for new polytopes.  A particularly interesting application of our results is a lower bound on the facial distance of the Birkhoff polytope, and a similar bound for more general {\em simplex-like} polytopes. 

Our paper is organized as follows. Section~\ref{sec.prelim} recalls the formal definition of facial distance and introduces some notation and terminology that we use throughout the paper.  Section~\ref{sec.product},  Section~\ref{sec.facet}, and Section~\ref{sec.linear} present our main developments, namely some calculus rules for products, facet-exposing directions, and linear transformations.  Section~\ref{sec.examples} illustrates our calculus rules on several popular polytopes.  Finally, Section~\ref{sec.conjecture} discusses  conjecture on the facial distance that is naturally suggested by the examples in Section~\ref{sec.examples}.

\section{Preliminaries}\label{sec.prelim}

Let $\R^n$ be endowed with a norm $\|\cdot\|$ and let $\cC \subseteq \R^n$ be a polytope.  Denote the sets of vertices and nonempty faces of $\cC$ by $\vertices(\cC)$ and $\faces(\cC)$ respectively.  For $F, G \in \faces(\cC)$ let 
\[
    \dist(F, G) = \min_{\xx\in F, \yy\in G} \|\xx-\yy\|.
\]
To ease notation, we will write $\cC[G]$ as shorthand for $\conv(\vertices(\cC)\setminus G)$ whenever $\cC$ is a polytope and $G\in \faces(\cC)$ with $G\ne \cC$. For $k\in \N$, we will write $[k]$ to denote the set $\{1,\dots,k\}$.
We next recall the definition of our main object of discussion, namely the {\em facial distance} of a polytope.  We also recall the definition of the much simpler and intuitive {\em diameter} of a polytope.

\begin{definition}[facial distance]\label{def.facial_distance}
Let $\R^n$ be endowed with a norm $\|\cdot\|$ and let $\cC \subseteq \R^n$ be a polytope. For $F\in \faces(\cC)$ the \emph{facial distance} from $F$ is defined as
\[
        \Phi(F,\cC):=\min_{G\in\faces(F)\atop G\subsetneq \cC} \dist(G, \cC[G]).
\]
We will write $\Phi(\cC)$ as shorthand for $\Phi(\cC,\cC)$.  We note that $\Phi(\cC,\cC)\le \Phi(F,\cC)$ for every $F\in\faces(\cC)$.  Furthermore, $\Phi(\cC,\cC)$ is exactly the same as the {\em pyramidal width} defined in~\cite{lacoste2015global} as shown in~\cite{pena2019polytope}.
\end{definition}

\begin{definition}[diameter]\label{def.facial_distance}
Let $\R^n$ be endowed with a norm $\|\cdot\|$ and let $\cC \subseteq \R^n$ be a polytope. The \emph{diameter} of $\cC$ is defined as
\[
\diam(\cC):=\max_{\xx,\yy\in \cC} \|\xx-\yy\|.
\]
\end{definition}

\section{Products}
\label{sec.product}
Theorem~\ref{thm.product} below gives a generic identity for the facial distance of the cartesian product of $k$ polytopes in terms of the facial distances of the individual components.  Throughout this section let $n_1,\dots,n_k\in \N$ be fixed and assume that the norms on $\R^{n_i}$ and on $\R^{n_1+\cdots+n_k} = \R^{n_1} \times \cdots \times \R^{n_k}$ are related via
\[
\|(\xx_1,\dots,\xx_k)\| = \vertiii{(\|\xx_1\|,\dots,\|\xx_k\|)}
\]
for some norm $\vertiii{\cdot}$ in $\R^k$, and some norms $\|\cdot\|$ on each of the $\R^{n_i},\; i=1,\dots,k$. Observe that $\vertiii{\cdot} = \|\cdot\|_2$ in the special case when each $\R^{n_i}$ and $\R^{n_1+\cdots+n_k}$ are endowed with the Euclidean norm.

We shall say that a norm $\vertiii{\cdot}$ on 
$\R^k$ is {\em reasonable} if $\vertiii{x} \le \vertiii{y}$ whenever $|x| \le |y|$ componentwise.  Notice that many popular norms on $\R^k$, including all $\ell_p$ norms for $p\in [1,\infty]$ are reasonable.  Observe also that $\vertiii{\cdot}$ is reasonable if and only if $\vertiii{\cdot}^*$ is reasonable.

The following identity is inspired by and generalizes~\cite[Theorem 4.5]{iommazzo2025linear}.

\begin{theorem}\label{thm.product} Suppose that for $i=1,\dots,k$ the sets $\cC_i \subseteq \R^{n_i}$ are polytopes  and $F_i \in \faces(\cC_i)$.  If the norm $\vertiii{\cdot}$ on $\R^k$ is reasonable then for $F:=F_1\times \cdots \times F_k$ and $\cC:=\cC_1\times \cdots \times \cC_k$ it holds that
\begin{equation}\label{eq.prod.dist}
\frac{1}{\Phi(F,\cC)} = \vertiii{\left(\frac{1}{\Phi(F_1,\cC_1)},\cdots,\frac{1}{\Phi(F_k,\cC_k)}\right)}^*.
\end{equation}
\end{theorem}

We note that the format of Theorem~\ref{thm.product} is similar to that of the following simpler and intuitive identity for the diameter.

\begin{proposition}\label{prop.product} 
Suppose that for $i=1,\dots,k$ the sets $\cC_i \subseteq \R^{n_i}$ are polytopes and $\cC:=\cC_1\times \cdots \times \cC_k$.  If the norm $\vertiii{\cdot}$ on $\R^k$ is reasonable then
\[
\diam(\cC) = \vertiii{\left(\diam(\cC_1),\dots,\diam(\cC_k)\right)}.
\]
\end{proposition}
\begin{proof}
This is a straightforward calculation.  
Since $\vertiii{\cdot}$ is reasonable, it follows that 
\[
\begin{aligned}
\diam(\cC) &= \max_{\xx,\yy\in \cC} \|\xx-\yy\| \\
&=\max_{\xx,\yy\in \cC}  \vertiii{\left(\|\xx_1-\yy_1\|,\dots,\|\xx_k-\yy_k\|\right)}\\
&=\vertiii{\left(\max_{\xx_1,\yy_1\in \cC_1} \|\xx_1-\yy_1\|,\dots,\max_{\xx_k,\yy_k\in \cC_k} \|\xx_k-\yy_k\|\right)}\\
&=\vertiii{\left(\diam(\cC_1),\dots,\diam(\cC_k)\right)}.
\end{aligned}
\]
\end{proof}
In contrast to Proposition~\ref{prop.product}, the proof of Theorem~\ref{thm.product} quite a bit more elaborate.  It relies on the following lemma, which is a slight modification of~\cite[Lemma 4.3]{iommazzo2025linear}.  

\begin{lemma}\label{lemma} Suppose $\cC \subseteq \R^{n},$ $F\in \faces(\cC)$.
Let 
\[
G:= \argmin_{G'\in\faces(F)\atop G'\subsetneq \cC} \dist(\affine(G'), \cC[G'])
\]
Then the point $\pp\in \affine(G)$ attaining $\dist(\affine(G), \cC[G])$ belongs to $G$.  In particular, 
\[
\Phi(F,C) = 
\dist(G, \cC[G]) = \dist(\affine(G), \cC[G]) = \min_{G'\in\faces(F)\atop G'\subsetneq \cC} \dist(\affine(G'), \cC[G']).
\]
\end{lemma}
\begin{proof}
The proof is nearly identical to the proof of~\cite[Lemma 4.3]{iommazzo2025linear}.  We include it here since the above statement is a slight modification of this previous result.

Proceed by contradiction.  Suppose $\pp\in \affine(G)$ and $\qq \in \cC[G]$ are such that 
$\|\pp-\qq\| = \dist(\affine(G), \cC[G])$ and $\pp \not \in G$.  Let $\tilde G$ be a facet of $G$ such that $\affine(\tilde G)$ separates $\pp$ from $G$ and let $\vv\in \vertices(G) \setminus \affine(\tilde G).$   Since  $\affine(\tilde G)$ has codimension one in $\affine(G)$ and the points $\pp,\vv\in\affine(G)$ are in opposite sides of this  hyperplane, it follows that the segment $[p,v]$ must intersect $\affine(\tilde G)$.  That is, $\tilde \pp := \lambda \pp + (1-\lambda)\vv \in\affine(\tilde G)$ for  some $\lambda\in (0,1)$.  Therefore  $\tilde \qq:=\lambda \qq + (1-\lambda)\vv \in \cC[\tilde G]$ and $\|\tilde \pp - \tilde \qq\| = \lambda \|\pp-\qq\|$ and consequently
\[
\dist(\tilde G, \cC[\tilde G])\le \|\tilde \pp - \tilde \qq\| = \lambda \|\pp-\qq\| < \dist(\affine(G), \cC[G]).
\]
This contradicts the choice of $G$. Thus it must be the case that $\pp\in G$.

For the second statement, observe that since $\pp\in G$ it follows that
\[
\Phi(F,\cC) \le 
\dist(G, \cC[G]) = \dist(\affine(G), \cC[G]) = \min_{G'\in\faces(F)\atop G'\subsetneq \cC} \dist(\affine(G'), \cC[G']).
\]
Since the above right-most expression is evidently less than or equal to $\Phi(F,\cC)$, it follows that all of the above expressions are the same.
\end{proof}

\begin{proof}[Proof of Theorem~\ref{thm.product}]
For $i=1,\dots,k$ let $G_i \in \faces(F_i)$ be such that 
\[
\delta_i := \Phi(F_i,\cC_i) = \dist(G_i,\cC_i[G_i]).
\]
Then $G := G_1\times \cdots\times G_k \in \faces(F)$. For $i\in[k]$ 
let $\pp_i \in G_i, \qq_i\in \cC_i[G_i]$ be such that $\|\pp_i-\qq_i\| = \delta_i$.  First we show that the left-hand-side in~\eqref{eq.prod.dist} is larger than or equal than the right-hand-side.  To that end, it suffices to construct $\pp\in G, \qq\in \cC[G]$ such that
\begin{equation}\label{eq.to.show}
\|\pp-\qq\|=\frac{1}{\vertiii{1/\boldsymbol{\delta}}^*}
\end{equation}
for $1/\boldsymbol{\delta}:=(1/\delta_1,\dots,1/\delta_k) \in \R^k$.
Indeed, if~\eqref{eq.to.show} holds for some $\pp\in G, \qq\in \cC[G]$ then 
\begin{equation}\label{lhs.less.rhs}
\frac{1}{\Phi(F,\cC)} \ge \frac{1}{\dist(G,\cC[G])} \ge \vertiii{1/\delta}^* = 
\vertiii{\left(\frac{1}{\Phi(F_1,\cC_1)},\cdots,\frac{1}{\Phi(F_k,\cC_k)}\right)}^*.
\end{equation}
The construction of $\pp\in G$ and $\qq\in \cC[G]$ is as follows.  Let
\[
\pp = (\pp_1,\dots,\pp_k), \qq = \lambda_1\qq^1+\cdots+\lambda_k\qq^k
\]
for $\qq^1,\dots,\qq^k\in \cC[G]$ and $\boldsymbol{\lambda} \in \Delta_{k-1}$ that we next describe. For $i\in[k]$ let $\qq^i \in \cC[G]$ be the vector obtained by changing the $i$-th block of $\pp\in G$ from $\pp_i\in G_i$ to $\qq_i\in \cC_i[G_i]$.  In other words,
\[
\qq^1 = (\qq_1,\pp_2,\dots,\pp_k), \qq^2 = (\pp_1,\qq_2,\pp_3,\dots,\pp_k), \dots, \qq^k = (\pp_1,\dots,\pp_{k-1},\qq_k).
\]
By definition of the dual norm, 

\begin{align*}
   \vertiii{1/\boldsymbol{\delta}}^* = \max_{\vertiii{z} \leq 1} \langle 1/\boldsymbol{\delta} , z\rangle.
\end{align*}
Thus, after rescaling, there exists a $\zz\in \R^k$ such that $ \langle 1/\boldsymbol{\delta} , \zz\rangle = 1$ and $\vertiii{\zz} = 1/\vertiii{1/\boldsymbol{\delta}}^*$. Since the norm $\vertiii{\cdot}$ is reasonable and $1/\delta > 0$ it follows that $\zz\ge 0$. For $i \in [k]$, let $\lambda_i := z_i / \delta_i$. Since $\zz\geq0$ and $1/\boldsymbol{\delta} >0$ it follows that $\boldsymbol{\lambda} \geq 0$. Also $\zz = (\lambda_1 \delta_1,\dots,\lambda_k\delta_k)$. Finally, $\boldsymbol{\lambda} \in \Delta_{k-1}$ because

\begin{align*}
    1 & = \langle 1/\boldsymbol{\delta} , \zz\rangle = \sum_{i=1}^k z_i / \delta_i = \sum_{i=1}^k \lambda_i.
\end{align*}
We next show that $\pp,\qq$ satisfy~\eqref{eq.to.show}.  Indeed,
\[
\|\pp-\qq\| = \|(\lambda_1(\pp_1-\qq_1),\dots,\lambda_k(\pp_k-\qq_k))\| = \vertiii{(\lambda_1 \delta_1,\dots,\lambda_k\delta_k)} = \vertiii{\zz} = \frac{1}{\vertiii{1/\boldsymbol{\delta}}^*}.
\]
Thus~\eqref{eq.to.show} holds and consequently so does~\eqref{lhs.less.rhs}.

Next we show that the left-hand-side in~\eqref{eq.prod.dist} is less than or equal than the right-hand-side.  To that end, it suffices to show that for all $G'\in \faces(F), G'\subsetneq \cC$ it holds that
\begin{equation}\label{eq.to.show.other}
\dist(\affine(G'),\cC[G']) \ge \frac{1}{\vertiii{1/\boldsymbol{\delta}}^*}.
\end{equation}
Indeed, if~\eqref{eq.to.show.other} holds for all $G'\in \faces(F), G'\subsetneq \cC$
then Lemma~\ref{lemma} yields
\begin{equation}\label{rhs.less.lhs}
\begin{aligned}
\frac{1}{\Phi(F,\cC)} &= \frac{1}{\displaystyle\min_{G'\in\faces(F)\atop G'\subsetneq \cC} \dist(\affine(G'), \cC[G'])} \\
&\le \vertiii{1/\boldsymbol{\delta}}^* \\&= 
\vertiii{\left(\frac{1}{\Phi(F_1,\cC_1)},\cdots,\frac{1}{\Phi(F_k,\cC_k)}\right)}^*.
\end{aligned}
\end{equation}
Suppose $G'\in \faces(F), G'\subsetneq \cC$. Then $G'= G'_1\times \cdots\times G'_k$ for 
some $G'_i \in \faces(F_i), \; i=1,\dots,k$. Since $G'\subsetneq \cC$, we must have $G'_i\subsetneq \cC_i$ for some $i=1,\dots,k$.  Without loss of generality assume $G'_i\subsetneq \cC_i$ for $i=1,\dots,j$ with $1\le j \le k$. For each $i=1,\dots,j$ let 
$\delta'_i = \dist(\affine(G'_i),\cC_i[G'_i])$ and let $\aa_i\in\R^{n_i}$ be such that $\|\aa_i\|^*=1$, $\ip{\aa_i}{\xx}$ is constant on $\affine(G'_i)$, and 
$
\ip{\aa_i}{\xx-\yy} \ge \delta'_i \text{ for all } \xx\in \affine(G'_i), \yy\in \cC_i[G'_i].
$
Let
\[
\aa := \frac{1}{\vertiii{(1/{\delta'_1},\dots, 1/{\delta'_j},0,\dots,0)}^*}
\cdot \left(\frac{\aa_1}{\delta'_1},\dots, \frac{\aa_j}{\delta'_j},\0,\dots,\0\right).
\]
The construction of $a$ ensures that $\|\aa\|^*=1$ and $\ip{\aa}{\xx}$ is constant on $\affine(G')$.  Furthermore, for any $\xx\in \affine(G')$ the  minimum of $\ip{\aa}{\xx-\yy}$ over $y\in\cC\setminus G'$ is attained at some $\vv \in \vertices(\cC\setminus G')$.  Thus $\vv_i \in \cC_i \setminus G'_i$ for some $i=1,\dots,j$ and consequently
\[
\begin{aligned}
\min\{\ip{\aa}{\xx-\yy}\;|\;  \yy\in\cC[G']\}  &\ge \frac{1}{\vertiii{(1/{\delta'_1},\dots, 1/{\delta'_j},0,\dots,0)}^*}\cdot \frac{\ip{\aa_i}{\xx_i-\vv_i}}{\delta'_i}\\
& \ge\frac{1}{\vertiii{(1/{\delta'_1},\dots, 1/{\delta'_j},0,\dots,0)}^*}.
\end{aligned}
\]
In particular, since $\|\aa\|^*=1$, it follows that for all $\xx\in \affine(G'), \; y\in\cC[G']$
\[
\|\xx-\yy\| \ge \ip{\aa}{\xx-\yy} \ge 
\frac{1}{\vertiii{(1/{\delta'_1},\dots, 1/{\delta'_j},0,\dots,0)}^*}
\]
and consequently
\[
\dist(\affine(G'),\cC[G']) \ge 
\frac{1}{\vertiii{(1/{\delta'_1},\dots, 1/{\delta'_j},0,\dots,0)}^*}.  
\]
Furthermore $\delta'_i \ge \delta_i = \Phi(F_i,\cC_i)$ for $i=1,\dots,j$ because $\delta'_i = \dist(G'_i,\cC_i[G'_i])$ and $G'_i\in \faces(F_i)$ with $G'_i\ne \cC_i$.  Therefore, since $\vertiii{\cdot}^*$ is reasonable, it follows that
\[
\begin{aligned}
\dist(\affine(G'),\cC[G']) &\ge 
\frac{1}{\vertiii{(1/{\delta'_1},\dots, 1/{\delta'_j},0,\dots,0)}^*}\\
&\ge 
\frac{1}{\vertiii{(1/{\delta_1},\dots, 1/{\delta_j},0,\dots,0)}^*}\\
&\ge\frac{1}{\vertiii{1/\boldsymbol{\delta}}^*}.
\end{aligned}
\]
Thus~\eqref{eq.to.show.other} holds and consequently so does~\eqref{rhs.less.lhs}.

\end{proof}

\section{Facet-exposing directions}
\label{sec.facet}
Theorem~\ref{thm.facets} below gives bounds on the facial distance of a polytope in term of {\em facet-exposing directions} and {\em gaps} described next.  In contrast to the facial distance, the latter quantities are often easily computable for polytopes with few facets or with a suitable polyhedral description as some of our examples in Section~\ref{sec.examples} illustrate.

Suppose $\cC\subseteq\R^n$ is a polytope and $F\in \faces(\cC), \; F\subsetneq \cC$.  We say that $\aa\in \R^n$ is a {\em face-exposing direction} of $F$ if $\|\aa\|^*=1$ and
\[
F = \argmin\{\ip{\aa}{\xx} \:|\; \xx\in \cC\}.
\]
When this is the case, define the {\em gap} of $\aa$ as follows
\[
\sigma(\aa) := \min\{\ip{\aa}{\vv} \; | \; \vv\in \vertices(\cC)\setminus F\} = \min\{\ip{\aa}{\xx} \; | \; \xx \in \cC[F]\}>0.
\]
Observe that $\dist(\affine(F),\cC[F]) \ge \sigma(\aa)$ for any face exposing direction of $F$ and the equality is attained for a judiciously chosen face-exposing direction.  We will say {\em facet-exposing direction} in lieu of face-exposing direction in the special case when $F$ is a facet.

Next, suppose $F_1,\dots,F_m$ are the facets of $\cC$. 
For $F\in \faces(\cC)$ let $I(F):=\{j \;|\; F\subseteq F_j\}$. Observe  that $I(F)\ne \emptyset$ whenever $F\subsetneq \cC$ and in that case 
 \[
 F = \bigcap_{j\in I(F)}F_j.
 \] Theorem~\ref{thm.facets} below gives bounds on the facial distance of $\cC$ in terms of  facet-exposing directions of the facets of $\cC$.

\begin{theorem}\label{thm.facets}
Suppose $\cC\subseteq\R^n$ is a polytope and $F_1,\dots,F_m$ are the facets of $\cC$ with facet-exposing directions $\aa_j, \; j=1,\dots,m$. Let 
$\sigma_j:=\sigma(\aa_j), \; j=1,\dots,m$.    Then for $F\in \faces(\cC)$ with $F\subsetneq \cC$
\[
\dist(\affine(F),\cC[F]) \ge \frac{1}{\|\sum_{j\in I(F)} \aa_j/\sigma_j\|^*}.
\]
Consequently for $F\in \faces(\cC)$ it holds that
\[
\Phi(F,\cC) \ge \min_{G\in\faces(F)\atop G\subsetneq \cC} \frac{1}{\|\sum_{j\in I(G)} \aa_j/\sigma_j\|^*}.
\]
\end{theorem}
\begin{proof}
The choice of $\aa_j$ and $\sigma_j$ implies that  $\|\aa_j\|^*=1, \; \ip{\aa_j}{\xx}$ is constant on $\affine(F_j)$, and $\ip{\aa_j}{\xx-\yy} \ge \sigma_j$ for all $\xx\in \affine(F_j)$ and all $\yy\in \conv(\vertices(\cC)\setminus F_j).$ 
Let 
\[
\aa := \sum_{j\in I(F)} \frac{a_j}{\sigma_j}.
\]
It follows that $\ip{\aa}{\xx}$ is constant on $\affine(F)=\bigcap_{j\in I(F)}\affine(F_j)$.  Furthermore, for $\vv\in \vertices(\cC)\setminus \affine(F) = \vertices(\cC)\setminus\bigcap_{j\in I(F)}\affine(F_j)$ there must exist at least one $j\in I(F)$ such that $\vv\not \in F_j$ and thus for all $\xx\in \affine(F) = \bigcap_{j\in I(F)}\affine(F_j)$ it holds that
\[
\ip{\aa}{\xx-\vv} \ge \frac{\ip{\aa_j}{\xx-\vv}}{\sigma_j} \ge 1.
\]
Since this holds for all $\xx\in \affine(F)$ and $\vv \in \vertices(\cC)\setminus \affine(F)$ it follows that
\[
\dist(\affine(F),\cC[F]) \ge \frac{1}{\|\aa\|^*} \ge \frac{1}{\|\sum_{j\in I(F)} \aa_j/\sigma_j\|^*}.
\]
\end{proof}

For $F\in \faces(\cC)$ let $J(F) := \{j \;|\; F\cap F_j\ne \emptyset\}$. Observe that $J(F)\ne \emptyset$ for all $F\in\faces(\cC)$.  In particular, $J(\cC) = \{1,\dots,m\}$.  The following corollary readily follows from Theorem~\ref{thm.facets} and the triangle inequality.

\begin{corollary}
Suppose $\cC\subseteq\R^n$ is a polytope with facets $F_1,\dots,F_m$ and $\delta_i := \dist(\affine(F_i),\cC[F_i])$ for $i=1,\dots,m.$ 
Then for each $F\in \faces(\cC)$ it holds that
\[
\Phi(F,\cC) \ge \frac{1}{\sum_{j\in J(F)} 1/\delta_j}.
 \]
In particular,
\[
\frac{1}{\sum_{j=1}^m 1/\delta_j} \le \Phi(\cC) \le \min\{\delta_j \;|\; j=1,\dots,m\}.
\]
\end{corollary}

\section{Linear transformations}
\label{sec.linear}
Throughout this section we shall assume that 
$\R^m$ and $\R^n$ be endowed with some norms,
 $\cC \subseteq \R^n$ be a polytope, $A\in \R^{m\times n}$, and $\cD = A(\cC)\subseteq \R^m$.   
A natural question is how the facial distances of $\cD$ and of $\cC$ are related.  
The following lemma provides some  tools to tackle that question.

\begin{lemma}\label{lemma.transformed}
Let $\cC \subseteq \R^n$ be a polytope, $A \in \R^{m \times n}$, $\cD = A(\cC)$. Then the following hold:
\begin{itemize}
\item[(a)] $\vertices(\cD) \subseteq A (\vertices(\cC))$.
\item[(b)] If $G\in \faces(\cD)$ then $A^{-1}(G)\cap \cC \in \faces(\cC)$.
\item[(c)] If $G\in \faces(\cD)$ and $G\subsetneq \cD$ then $F:=A^{-1}(G)\cap \cC\subsetneq \cC$ and 
\begin{equation}\label{eq.transformed.dist}
\dist(G, \cD[G]) = \min\{\|A(\yy-\xx)\| \; |\; \xx \in F, \yy \in \cC[F]\}.
\end{equation}
\end{itemize}

\end{lemma}
\begin{proof}
\begin{itemize}
\item[(a)] Observe that $\cD = A(\cC) = A(\conv(\vertices(\cC))) = \conv(A(\vertices(\cC)))$.  Therefore $\vertices(\cD) \subseteq  A(\vertices(\cC))$ since $\vertices(\cD)$ is the smallest set of points whose convex hull is $\cD$.
\item[(b)] Since $G \in \faces(\cD)$, there exists $\aa \in \R^m$ such that
\[
G = \argmin_{\yy\in \cD} \ip{\aa}{\yy} = A\left(\argmin_{\xx\in \cC} \ip{\aa}{A\xx}\right) = A\left(\argmin_{\xx\in \cC} \ip{A\transp\aa}{\xx}\right).
\] 
It thus follows that 
\[
F=A^{-1}(G)\cap \cC = \argmin_{\xx\in \cC} \ip{A\transp\aa}{\xx},
\]
and consequently $F\in \faces(\cC)$.
\item[(c)] Equation~\eqref{eq.transformed.dist} is an immediate consequence of parts (a) and (b).  Indeed, parts (a) and (b) imply that for any two $\zz,\ww\in \R^m$ it holds that 
$\zz \in G$ and $\ww \in \cD[G]$ if and only if $\zz = A\xx$ and $\ww = A \yy$ for some 
$\xx \in F$ and $\yy \in \conv(\vertices(\cC) \setminus F)$. Equation~\eqref{eq.transformed.dist} thus follows.
\end{itemize}

\end{proof}

\begin{corollary} Let $\cC \subseteq \R^n$ be a polytope, $A \in \R^{n \times n}$ be non-singular,  and $G\in \faces(\cD)$. Then for $F := A^{-1}(G)\cap \cC \in \faces(\cC)$ it holds that
\begin{equation}\label{eq.dist.transformed}
 \frac{\Phi(F,\cC)}{\|A^{-1}\|} \leq \Phi(G,\cD) \leq  \|A\| \Phi(F,\cC).
\end{equation}
\end{corollary}
\begin{proof}
Equation~\eqref{eq.transformed.dist} in Lemma~\ref{lemma.transformed} implies that for all
$G'\in \faces(G)$  with $G'\ne \cC$ and $F':= A^{-1}(G')\cap \cC$
\[
\begin{aligned}
\dist(G',\cD[G']) &= \min\{\|A(\yy-\xx)\| \; |\; \xx \in F', \yy \in \cC[F']\} \\
&= \frac{1}{\|A^{-1}\|} \cdot \min\{\|A^{-1}\| \cdot \|A(\yy-\xx)\| \; |\; \xx \in F', \yy \in \cC[F']\} \\
&\ge \frac{1}{\|A^{-1}\|}\min\{\|\yy-\xx\| \; |\; \xx \in F', \yy \in \cC[F']\} \\
& = \frac{\dist(F',\cC[F'])}{\|A^{-1}\|} \\
& \ge \frac{\Phi(F,\cC)}{\|A^{-1}\|}.
\end{aligned}
\]
Since this holds for all $G'\in \faces(G)$  with $G'\ne \cC$, the lower bound of~\eqref{eq.dist.transformed} follows.

Similarly,
\[
\begin{aligned}
\dist(G',\cD[G']) &= \min\{\|A(\yy-\xx)\| \; |\; \xx \in F', \yy \in \cC[F']\} \\
&\leq \|A\| \cdot \min\{\|(\yy-\xx)\| \; |\; \xx \in F', \yy \in \cC[F']\} \\
& = \|A\| \dist(F',\cC[F'])
\end{aligned}
\]
Since $A$ is nonsingular, the map $G' \mapsto F' := A^{-1}(G')\cap \cC$ is a bijection between $\faces(G)$ and $\faces(F)$. Thus, 
\begin{align*}
   \min_{\substack{G'\in\faces(G)\\ G'\neq G}} \dist(G',\cD[G']) & \leq \min_{\substack{G'\in\faces(G)\\ G'\neq G}} \|A\| \dist(A^{-1}(G')\cap \cC,\cC[A^{-1}(G')\cap \cC]) \\
   & = \|A\| \min_{\substack{F'\in\faces(F)\\ F'\neq F}} \dist(F',\cC[F']) \\
   & =\|A\| \Phi(F, \cC).
\end{align*}
Thus, the upper bound of~\eqref{eq.dist.transformed} follows.

\end{proof}

\section{Examples}
\label{sec.examples}
%: simplex, $\ell_1$ ball, $\ell_{\infty}$ ball, Birkhoff polytope}  

The examples below illustrate how the above calculus rules can be used to bound or compute the facial distance.  For the sake of intuition, throughout this section we assume that the norm in the underlying space is the canonical Euclidean norm $\|\cdot\|_2$.

\begin{example}[standard simplex]\label{simplex}
Suppose %$\R^n$ is endowed with the Euclidean norm, 
$n > 1$, and 
\[\cC=\{\xx\in \R^n \; |\; \xx\ge 0, \ip{\mathbf{1}}{\xx} = 1\}.\]
The vertices of $\cC$ are the unitary vectors $\ee_i, \; i=1,\dots,n$.  The
facets of $\cC$ are
\[
F_i = \{\xx\in \cC \; |\; x_i = 0\} \;  \text{ for } \; i=1,\dots,n.
\]
Thus the directions $\ee_i, \; i=1,\dots,n$ are face-exposing directions of the facets of $\cC$ and  $\sigma(\ee_i) = 1$ for $i=1,\dots,n$. 
For any $F\in \faces(\cC)$ with $F\ne \cC$ we have
\[
I(F) = \{i \in [n] \; |\; x_i = 0 \text{ for all } \xx\in F\}
\]
and so $|I(F)|\le n-1$.
%Hence for all $G\in \faces(F)$ it follows that $\|\sum_{j\in I(G)} \ee_j/\sigma_j\|^* = \sqrt{|I(G)|} \le \sqrt{|I(F)|}$.
Therefore Theorem~\ref{thm.facets} implies that for all $F\in\faces(\cC)$ with $F\subsetneq \cC$ 
\[
\dist(\affine(F),\cC[F]) \ge \frac{1}{\|\sum_{j\in I(F)} \ee_j\|^*} = \frac{1}{\sqrt{|I(F)|}}\ge \frac{1}{\sqrt{n-1}},
\]
and consequently Lemma~\ref{lemma} implies that
\[
\Phi(\cC) \ge \frac{1}{\sqrt{n-1}}. 
\]
We note that this lower bound is within a small constant of the exact value of $\Phi(\cC)$ previously computed in~\cite{lacoste2015global,pena2019polytope}:
\[
\Phi(\cC) = \left\{ \begin{array}{ll} \frac{2}{\sqrt{n}} & \text{ if } n \text{ is even}\\
\frac{2}{\sqrt{n-1/n}}& \text{ if } n \text{ is odd.}
\end{array} \right.
\]

\end{example}

\begin{example}[$\ell_1$ unit ball]\label{ex.l1.ball}
Suppose $n\ge 1$ and %$\R^n$ is endowed with the Euclidean norm and 
\[\cD=\{\xx\in \R^n \; |\; \|\xx\|_1 \le 1\} = A(\cC)\]
 where $A = \matr{I & -I}\in \R^{n\times 2n} $ and $\cC \subseteq \R^{2n}$ is the standard simplex in $\R^{2n}$, that is,
\[\cC=\{\xx\in \R^{2n}  \; |\; \xx\ge 0, \ip{\mathbf{1}}{\xx} = 1\}.\]
The set of vertices of $\cD$ is $\{\pm \ee_j  \; |\; j \in [n]\}$.  

Let $G\in \faces(\cD)$.  We next compute $\dist(\affine(F),\cD[G])$.  Since $\cD$ and the Euclidean norm are invariant under changes of signs and permutations of any set of coordinates, we can assume without loss of generality that $G$ is of the form $\conv\{\ee_1,\dots,\ee_k\} = \{\zz\in\cD: \zz_{k+1:n} = \0\}$ for some $k\le n$. 
It thus follows that $F = A^{-1}(G) = \{\xx\in \cC\;|\;  \xx_{k+1:2n} = \0\}$ and hence $\cC[F] = \{\yy\in \cC\;|\;  \yy_{1:k} = \0\}$. Lemma~\ref{lemma.transformed} implies that
\[
\begin{aligned}
\dist(G,\cD[G]) &= \min\{\|A(\xx-\yy)\| \; |\; \xx \in F, \yy \in \cC[F]\} \\
&= \min\{\|\matr{\xx_{1:n} - \xx_{n+1:2n}} - \matr{\yy_{1:n} - \yy_{n+1:2n}}\|  \; |\; \xx\in F, \yy\in \cC[F]\} \\
&= 
\min\left\{ \left\|\matr{\xx_{1:k} \\ \0_{k+1:n}} - \matr{-\yy_{n+1:n+k} \\ \yy_{k+1:n} - 
\yy_{n+k+1:2n}}\right\|  \; |\; \xx\in F, \yy\in \cC[F] \right\} \\
&= 
\min\left\{ \left\|\matr{\xx_{1:k} +\yy_{n+k+1:2n} \\ \yy_{k+1:n} - 
\yy_{n+k+1:2n}}\right\| \; |\; \xx\in F, \yy\in \cC[F] \right\}.
\end{aligned}
\]
A simple calculation then shows that
\[
\dist(G,\cD[G]) = \left\{ \begin{array}{ll} \min\{\|\xx_{1:k}\| \; |\; \xx\in F\} = \frac{1}{\sqrt{k}} & \text{ if } k<n\\
\min\{\|\xx_{1:n}+\yy_{n+1:2n}\| \; |\; \xx\in F, \yy\in \cC[F]\} = \frac{2}{\sqrt{n}}& \text{ if } k=n.
\end{array} \right.
\]
Consequently,
\[
\Phi(\cD) = \frac{1}{\sqrt{n-1}}.
\]
\end{example}

\begin{example}[$\ell_{\infty}$ unit ball]
Suppose $n\ge 1$ and %$\R^n$ is endowed with the Euclidean norm and 
\[
\cC=\{\xx\in \R^n  \; |\; \|\xx\|_\infty \le 1\} = \cC_1\times\cdots\times \cC_n
\]
where $\cC_i = [-1,1]$ for each $i=1,\dots,n$.  In this case $\Phi(\cC_i) = 2$ for $i=1,\dots,n.$ Thus  Theorem~\ref{thm.product} implies that
\[
\Phi(\cC) = \frac{1}{\|\matr{1/2,\dots,1/2}\|_2} = \frac{1}{\sqrt{n}/2} = \frac{2}{\sqrt{n}}.
\]

\end{example}

\begin{example}[Birkhoff polytope]\label{birkhoff}
Suppose $\R^{n\times n}$ is endowed with the Frobenius norm and $\cC$ is the set of $n \times n$ doubly stochastic matrices, that is,
\[
\cC = \left\{\xx \in \R^{n\times n} \,|\, \xx\ge \mathbf{0}, \sum_{i=1}^n x_{ij} = 1  \text{ and } \sum_{j=1}^n x_{ij} = 1 \text{ for } i,j\in [n]\right\}.
\]
We will focus on the case $n\ge 3$ since for the cases $n=1$ and $n=2$ the Birkhoff polytope is respectively a singleton and a segment and thus of limited interest.

It is known that the vertices of $\cC$ are the set of permutation matrices and the facets of $\cC$ are of the form
\[
F_{ij} = \{\xx\in \cC  \; |\; x_{ij} = 0\} \text{ for } (ij)\in [n]\times[n].
\]
Thus the directions $\ee_i\ee_j\transp$ for $(ij)\in [n]\times[n]$
 are face-exposing directions of the facets of $\cC$ and  $\sigma(\ee_i\ee_j\transp) = 1$ for $(ij)\in [n]\times[n]$.
For any $F\in\faces(\cC)$ with $F\ne \cC$ we have 
\[
I(F) = \{(ij) \in [n]\times[n]  \; |\; x_{ij} = 0 \text{ for all } \xx\in F\}.
\]
and so $|I(F)|\le n^2-1$.
%Hence for all $G\in \faces(F)$ it follows that $\|\sum_{j\in I(G)} \ee_j/\sigma_j\|^* = \sqrt{|I(G)|} \le \sqrt{|I(F)|}$.
Therefore Theorem~\ref{thm.facets} implies that for all $F\in\faces(\cC)$ with $F\subsetneq \cC$ 
\[
\dist(\affine(F),\cC[F]) \ge \frac{1}{\|\sum_{j\in I(F)} \ee_j\|^*} = \frac{1}{\sqrt{|I(F)|}}\ge \frac{1}{\sqrt{n^2-1}},
\]
and consequently Lemma~\ref{lemma} implies that
\[
\Phi(\cC) \ge \frac{1}{\sqrt{n^2-1}}.
\]

\end{example}

The bounds in Example~\ref{simplex} and Example~\ref{birkhoff} generalize to the broader class of {\em simplex-like} polytopes introduced by Garber and Meshi~\cite{garber2016linear}.

\begin{definition}
A polytope $\cC\subseteq \R^n$ is {\em simplex-like} if it is of the form $
\{\xx\in \R^n  \; |\; A \xx = \bb, \xx \ge \0\}$ and $\vertices(\cC) \subseteq \{0,1\}^n$.
\end{definition}

Example~\ref{ex.simplex.like} relies on the following property of polytopes with a polyhedral description.

\begin{proposition}
\label{prop.facets}
Suppose $\cC\subseteq \R^n$ is a polytope of the form
$\cC = \{\xx\in \R^n  \; |\; A \xx = \bb, \xx \ge \0\}$.  Then each facet of $\cC$ has a face-exposing direction of  of the form $\ee_j$ for some $j\in [n]$.
\end{proposition}
\begin{proof}
Suppose $F\subseteq \cC$ is a facet.  Let 
\[J(F) :=\{i\in [n]\;|\;  \ip{\ee_i}{\xx} = 0 \text{ for all } \xx\in F\}.\]
In other words, $J(F)\subseteq[n]$ is the set of inequalities that are active at all points in $F$.  Since $F$ is a facet of $\cC$, in particular $F\ne \cC$ and thus there exists $j\in J(F)$ such that $\ip{\ee_j}{\xx} > 0$ for some $\xx\in \cC\setminus F$.  Since $\cC\subseteq \{\xx\in \cC\;|\; \ip{\ee_j}{\xx} \ge 0\}$, to show that $\ee_j$ is a facet-exposing direction for $F$ it suffices to show that
\begin{equation}\label{eq.to.finish}
F =\{\xx\in \cC\;|\; \ip{\ee_j}{\xx} = 0\}.
\end{equation}
The construction of $J(F)$ and the fact that $j\in J(F)$ guarantee that $F\subseteq \{\xx\in \cC\;|\; \ip{\ee_j}{\xx} = 0\}$.  Furthermore, the choice of $j\in J(F)$ guarantees that $ \{\xx\in \cC\;|\; \ip{\ee_j}{\xx} = 0\}$ is a proper face of $\cC$.  Since $F\subseteq \{\xx\in \cC\;|\; \ip{\ee_j}{\xx} = 0\}$ and $F$ is a facet of $\cC$, these two sets must be the same, that is,~\eqref{eq.to.finish} holds.
\end{proof}

\begin{example}\label{ex.simplex.like}
Suppose $\cC\subseteq \R^n$ is a {\em simplex-like} polytope.  Proposition~\ref{prop.facets} implies that each facet of $\cC$ has a face-exposing direction of  of the form $\ee_j$ for some $j\in [n]$ and in addition $\sigma(\ee_j) = 1$ because $\vertices(\cC)\subseteq \{0,1\}^n.$  For $F\in\faces(\cC)$, the size of the set $I(F)$ is the number of zeros of the most dense point in $F$.  In other words, $|I(F)|$ is a measure of the sparsity of $F$.  In particular,  $|I(F)| \le n$ for all $F\in\faces(\cC)$. 
Therefore Theorem~\ref{thm.facets} implies that for all $F\in\faces(\cC)$ with $F\subsetneq \cC$ 
\[
\dist(\affine(F),\cC[F]) \ge \frac{1}{\|\sum_{j\in I(F)} \ee_j\|^*} = \frac{1}{\sqrt{|I(F)|}}\ge \frac{1}{\sqrt{n}},
\]
and consequently Lemma~\ref{lemma} implies that
\[
\Phi(\cC) \ge \frac{1}{\sqrt{n}}.
\]

\end{example}

\section{Upper bound conjecture on the facial distance}
\label{sec.conjecture}

Some straightforward calculations shows that in the examples in Section~\ref{sec.examples} the ratio $\Phi(\cC)/\diam(\cC)$ is bounded below as follows: $\sqrt{2}/\sqrt{n}$ for the standard simplex, $1/(2\sqrt{n})$ for the $\ell_1$ unit ball in $\R^n$, $1/n$ for the $\ell_\infty$ ball in $\R^n$, $1/(n\sqrt{2n})$ for the $n\times n$ Birkhoff polytope, and $1/n$ for a simplex-like polytope in $\R^n$.
In addition, lower complexity results for the Frank-Wolfe algorithm suggest that $\Phi(\cC)/\diam(\cC)$ should be bounded above.  The following conjecture thus arises.
\begin{conjecture}\label{the.conj}
Suppose $\cC$ is a polytope of dimension $n$ and the ambient space is endowed with the Euclidean norm.  Then
\[
\frac{\Phi(\cC)}{\diam(\cC)} \le \frac{C}{\sqrt{n}}.
\]
for some constant $C>0$.
\end{conjecture}
In addition to the examples discussed in Section~\ref{sec.examples}, the next two partial results provide support for this conjecture.

\begin{proposition}\label{prop.upper.simplex}
Suppose $\cC$ is an $n$-dimensional simplex and $n$ is an odd number.  Then
\begin{equation}\label{eq.upper.bound}
\frac{\Phi(\cC)}{\diam(\cC)} \le \frac{2\sqrt{2}}{\sqrt{n+1}}.
\end{equation}
\end{proposition}
\begin{proof}
Since $\cC$ is an $n$-dimensional simplex, it is of the form $\cC = \conv\{\xx_1,\dots,\xx_{n+1}\}$ where $\xx_1,\dots,\xx_{n+1}$ are affine independent.
Furthermore, by scaling and shifting if needed, without loss of generality we can assume that $\diam(\cC) = 1$ and $\xx_1+ \cdots + \xx_{n+1} = \0$.  

Since $\cC$ is a simplex, it follows that every  $F\in\faces(\cC)$ is of the form $F=\conv\{\xx_i \;|\; i\in I\}$ for some nonempty $I\subseteq [n+1]$ and in that case $\cC[F] = \conv\{\xx_j \;|\; j\in [n+1]\setminus I\}$.  Therefore to prove~\eqref{eq.upper.bound} it suffices to show the following claim.

\medskip

\noindent
{\bf Claim.} There exists $I\subseteq [n+1]$ with $|I| = (n+1)/2$ such that
\begin{equation}\label{eq.claim}
\left\|\sum_{i\in I} \xx_i - \sum_{j\in [n+1]\setminus I} \xx_j \right\|^2 \le 2(n+1).
\end{equation}
Indeed, the above claim implies that for $F = \conv\{\xx_i \;|\; i\in I\}$ it holds that
\[
\begin{aligned}
\dist(F,\cC[F]) &\le \left\|\frac{2}{n+1}\sum_{i\in I} \xx_i - \frac{2}{n+1}\sum_{j\in [n+1]\setminus I} \xx_j \right\|\\& \le \frac{2}{n+1} \cdot \sqrt{2(n+1)} \\
&\le \frac{2\sqrt{2}}{\sqrt{n+1}},
\end{aligned}
\]
and hence $$\frac{\Phi(\cC)}{\diam(\cC)} \le \dist(F,\cC[F]) \le \frac{2\sqrt{2}}{\sqrt{n+1}}.$$
We next prove the claim.  To that end, we proceed in two steps.

\medskip
\noindent
{\em Step 1:}  Because
$\diam(\cC) = 1$ and $\xx_1+ \cdots + \xx_{n+1} = \0$, it holds that 
\begin{equation}\label{eq.ineqs}
\|\xx_i\| \le 1 \text{ for all } i\in[n+1] \text{ and } \sum_{i\ne j} \ip{\xx_i}{\xx_j} \le 0.
\end{equation}
For the first inequality in~\eqref{eq.ineqs}, proceed by contradiction.  Suppose $\|\xx_i\| > 1$ for some $i\in [n+1]$.  Then 
\begin{equation}\label{eq.pair}
\|\xx_i - \frac{1}{n} \sum_{j\ne i} \xx_j\|^2 \ge \|\xx_i\|^2 - \frac{2}{n}  
\sum_{j\ne i} \ip{\xx_i}{\xx_j}.
\end{equation}
But $0 = \|\xx_i + \sum_{j\ne i} \xx_j\|^2 = \|\xx_i\|^2 + \|\sum_{j\ne i} \xx_j\|^2 + 2\sum_{j\ne i} \ip{\xx_i}{\xx_j}$ and so $\sum_{j\ne i} \ip{\xx_i}{\xx_j} \le -\|\xx_i\|^2 < 0$.  Thus~\eqref{eq.pair} implies that
\[
\|\xx_i - \frac{1}{n} \sum_{j\ne i} \xx_j\|^2 > \|\xx_i\|^2 > 1
\]
which contradicts the assumption $\diam(\cC) = 1$.

For the second inequality in~\eqref{eq.ineqs}, observe that $0 = \|\sum_i \xx_i\|^2 = \sum_i \|\xx_i\|^2 + 2\sum_{i\ne j} \ip{\xx_i}{\xx_j}$ and so $\sum_{i\ne j} \ip{\xx_i}{\xx_j} = -  \sum_i \|\xx_i\|^2/2 \le 0.$

\medskip
\noindent
{\em Step 2:} Let
\[
\delta:= \min_{I\subseteq[n+1] \atop |I|=(n+1)/2} \left\|\sum_{i\in I} \xx_i - \sum_{j\in [n+1]\setminus I} \xx_j \right\|^2.
\]
we next show that $\delta\le 2(n+1)$
and thus~\eqref{eq.claim} must hold for some $I\subseteq[n+1]$ with $|I|=(n+1)/2$.

To that end, observe that by adding over all $I\subseteq[n+1]$ with $|I|=(n+1)/2$ we get
\begin{equation}\label{eq.simplex.bound}
\begin{aligned}
{(n+1) \choose (n+1)/2} \delta & \le 
\sum_{I\subseteq[n+1] \atop |I|=(n+1)/2} \left\|\sum_{i\in I} \xx_i - \sum_{j\in [n+1]\setminus I} \xx_j \right\|^2\\
& \le {(n+1) \choose (n+1)/2} \left( \sum_{i} \|\xx_i\|^2 - 2\sum_{i\ne j}\ip{\xx_i}{\xx_j} \right)\\
& = {(n+1) \choose (n+1)/2} \left( 2\sum_{i} \|\xx_i\|^2 - \|\sum_{i}\xx_i \|^2 \right)\\
& = 2{(n+1) \choose (n+1)/2} \sum_{i} \|\xx_i\|^2.
\end{aligned}
\end{equation}
The second step above holds because $\sum_{i\ne j} \ip{\xx_i}{\xx_j} \le 0$.
Inequality~\eqref{eq.simplex.bound} in turn implies that
\[
\delta \le 2\sum_{i} \|\xx_i\|^2 \le 2(n+1).
\]

\end{proof}

\begin{proposition}
Suppose that for $i=1,\dots,k$ the sets $\cC_i \subseteq \R^{n_i}$ are polytopes and let $\cC:=\cC_1\times \cdots \times \cC_k$.  If each $\R^{n_i}$ and $\R^{n_1+\cdots+n_k}$ are endowed with the Euclidean norms then
\begin{equation}\label{eq.ratio}
\frac{\diam(\cC)}{\Phi(\cC)} = \|\left(\diam(\cC_1),\dots,\diam(\cC_k)\right)\| \cdot
\left\|\left(\frac{1}{\Phi(\cC_1)},\cdots,\frac{1}{\Phi(\cC_k)}\right)\right\|.
\end{equation}
In particular if each $\frac{\Phi(\cC_i)}{\diam(\cC_i)} \le \frac{C}{\sqrt{n_i}}$ then $\frac{\Phi(\cC)}{\diam(\cC)} \le \frac{C}{\sqrt{n_1+\cdots+n_k}}$.
\end{proposition}
\begin{proof}
Identity~\eqref{eq.ratio} is an immediate consequence of Proposition~\ref{prop.product} and Theorem~\ref{thm.product}. The second statement  follows from Cauchy-Schwarz inequality and the concavity of the square root function:
\[
\begin{aligned}
\frac{\diam(\cC)}{\Phi(\cC)} &= \|\left(\diam(\cC_1),\dots,\diam(\cC_k)\right)\| \cdot
\left\|\left(\frac{1}{\Phi(\cC_1)},\cdots,\frac{1}{\Phi(\cC_k)}\right)\right\|\\
&\ge \sum_{i=1}^k \frac{\diam(\cC_i)}{\Phi(\cC_i)}\\
&\ge \sum_{i=1}^k \frac{\sqrt{n_i}}{C} \\
&\ge \frac{\sqrt{n_1+\cdots+n_k}}{C}.
\end{aligned}
\]
\end{proof}

In spite of the above two promising results, we have not been able to prove Conjecture~\ref{the.conj}.  It is natural to try to leverage Proposition~\ref{prop.upper.simplex} via an induction argument: The initial inductive step would be Proposition~\ref{prop.upper.simplex}.  This proposition shows the conjecture for any polytope $\cC$ of dimension $\dim(\cC) = n$ with the minimal number of vertices, namely $n+1$.  The induction would then carry on if we could show that the ratio $\Phi(\cC)/\diam(\cC)$ does not increase as the number of vertices of $\cC$ increases while $\dim(\cC)$ remains constant.  Unfortunately, the second step of this argument fails as illustrated by the following example.

\begin{example}\label{ex.counter.extra}
Suppose $0 \le \epsilon \le 1$ and 
consider the following distorted $\ell_1$ unit ball in $\R^3$:
\[
\cC = \conv\{\ee_1,-\ee_1,\ee_2 + \epsilon \ee_1,-\ee_2 + \epsilon \ee_1,\ee_3 - \epsilon \ee_1,-\ee_3 - \epsilon \ee_1\}.
\]
Observe that this is exactly the $\ell_1$ unit ball when $\epsilon = 0$. 
In that case Example~\ref{ex.l1.ball} shows that $\Phi(\cC)=1/\sqrt{2}$.  It thus follows that $\Phi(\cC)\approx 1/\sqrt{2}$ when $0 < \epsilon \ll 1$.  
Now consider the  polytope obtained after we remove one of the vertices of $\cC$:
\[
\cC' = \conv\{-\ee_1,\ee_2 + \epsilon \ee_1,-\ee_2 + \epsilon \ee_1,\ee_3 - \epsilon \ee_1,-\ee_3 - \epsilon \ee_1\}. 
\]
A simple calculation shows that when $0 < \epsilon \ll 1$ it holds that 
\[\Phi(\cC') \le \dist(F,\cC'[F]) = 2\epsilon \ll \Phi(\cC)\]
 for $F=[\ee_2 + \epsilon \ee_1,-\ee_2 + \epsilon \ee_1]\in\faces(\cC')$.  As a consequence $\Phi(\cC')/\diam(\cC') \ll \Phi(\cC)/\diam(\cC)$. This example shows that $\Phi(\cC)/\diam(\cC)$ may increase drastically when the number of vertices increases while $\dim(\cC)$ remains constant.
\end{example}

We also have not been able to prove Conjecture~\ref{the.conj} when $\cC$ is the Birkhoff polytope.  We hope Conjecture~\ref{the.conj} and the other developments in this paper motivate future work on a better understanding of the facial distance.

\begin{remark}
%We note that 
Example~\ref{ex.counter.extra} disproves a conjecture formulated by 
Lacoste-Julien and Jaggi in~\cite[Sect 3.1]{lacoste2015global}.  They conjectured that the  pyramidal
width (facial distance) is non-increasing when a vertex is added to a polytope, as long as the vertices of the old polytope remain vertices of the new one.  The very recent manuscript~\cite{zhao2026pyramidal} also provides a counterexample to the same conjecture.
\end{remark}

%\medskip

%In addition, it is true in the special case when $\cC$ is a simplex, as the following proposition holds.  Furthermore, the conjecture is also consistent with the 

%\bibliographystyle{plain}
%\bibliography{bibliography}

\end{document}